\documentclass[a4paper,11pt,reqno]{amsart}

\usepackage[utf8]{inputenc}
\usepackage[%
  style=trad-plain,
  backend=biber,
  sorting=none,
  url=false,
  doi=false,
  isbn=false,
  eprint=false,
  giveninits=true
]{biblatex}
\usepackage{amsmath,amsthm,amsfonts,amssymb}
\usepackage{graphicx}
\usepackage{verbatim}
\usepackage{enumitem}
 \setlist[enumerate,1]{label=(\arabic*)}
 \setlist[enumerate,2]{label=(\alph*)}
\usepackage{amssymb}
\usepackage{mathtools}
\usepackage{esint}
\usepackage{verbatim}
\usepackage{xcolor}
\definecolor{darkred}{rgb}{0.8,0,0}
\definecolor{darkblue}{rgb}{0,0,0.55}
\definecolor{darkgreen}{rgb}{0,0.39,0}
\usepackage[normalem]{ulem}
\usepackage{csquotes}
  \MakeAutoQuote{«}{»}

\usepackage{tikz}
\usetikzlibrary{shapes,positioning,intersections,quotes}
\newtheorem{theorem}{Theorem}[section]

\newtheorem{lemma}[theorem]{Lemma}

\newtheorem{definition}[theorem]{Definition}

\theoremstyle{definition}
\newtheorem{model}[theorem]{Model}
\newtheorem{remark}[theorem]{Remark}
\newtheorem*{remark*}{Remark}

\newcommand{\eps}{\varepsilon}
\newcommand{\epsk}{\varepsilon_k}

\newcommand{\R}{\mathbb{R}}
\newcommand{\N}{\mathbb{N}}

\newcommand{\sphere}{\mathbb{S}}
\newcommand{\calA}{\mathcal{A}}

\newcommand{\calD}{\ensuremath{\mathcal{D}}}
\newcommand{\calE}{\mathcal{E}}
\newcommand{\calF}{\mathcal{F}}
\newcommand{\calG}{\ensuremath{\mathcal{G}}}
\newcommand{\calH}{\mathcal{H}}
\newcommand{\calL}{\mathcal{L}}

\newcommand{\calP}{\mathcal{P}}

\newcommand{\calU}{\mathcal{U}}
\newcommand{\calV}{\mathcal{V}}

\newcommand{\calZ}{\mathcal{Z}}
\newcommand{\bldZ}{\mathbf{Z}}
\newcommand{\Mass}{\ensuremath{{\sf\mathbf{M}}}}

\newcommand{\BV}{\mathrm{BV}}
\newcommand{\weakto}{\rightharpoonup}

\DeclareMathOperator{\dist}{dist}
\DeclareMathOperator{\sdist}{sdist}

\DeclareMathOperator{\spt}{spt}
\DeclareMathOperator{\loc}{loc}
\DeclareMathOperator{\range}{range}

\usepackage[colorlinks=true]{hyperref} 
\hypersetup{linkcolor=darkblue,citecolor=darkgreen}

\newcommand{\me}{M_{\eps}}

\newcommand{\cek}{c\sqrt{\epsk}}
\newcommand{\och}{\overline{\chi}_k}
\newcommand{\ch}{\chi_{\eps}}

\begin{document}
\author{Jakob Fuchs}
\address[Jakob Fuchs]{Department of Mathematics,
Technische Universität Dortmund}
\email{jakob.fuchs@tu-dortmund.de}
\author{Matthias Röger}\address[Matthias Röger]{Department of Mathematics,
Technische Universität Dortmund}
\email{matthias.roeger@tu-dortmund.de}
\title{Mathematical analysis of a mesoscale model for multiphase membranes}

\begin{abstract}
In this paper we introduce a mesoscale continuum model for membranes made of two different types of amphiphilic lipids.
The model extends work by Peletier and the second author [\emph{Arch.~Ration.~Mech.~Anal.~193}, 2009] for the one-phase case.
We present a mathematical analysis of the asymptotic reduction to the macroscale when a key length parameter becomes arbitrarily small.
We identify two main contributions in the energy: one that can be connected to bending of the overall structure and a second that describes the cost of the internal phase separations.
We prove the $\Gamma$-convergence towards a perimeter functional for the phase separation energy and construct, in two dimensions, recovery sequences for the convergence of the full energy towards a 2D reduction of the Jülicher--Lipowsky bending energy with a line tension contribution for phase separated hypersurfaces.
\\[2ex]\noindent%
\\[2ex]\noindent%
\end{abstract}

\maketitle

\tableofcontents

\section{Introduction}\label{sec:1}
Phase separation processes play an important role in various applications and have been a key motivation for a number of mathematical fields, such as the Calculus of Variations and Geometric Measure theory.

The analysis of phase separation models is rather advanced in physics and materials science.
Notably, well founded atomistic and mesoscale variational models have been related to macroscopic \emph{sharp interface models} by rigorous mathematical convergence results.

Macroscopic descriptions of a separation of a given domain $\Omega\subset\R^d$ in two phases can be formulated in terms of phase indicator function $u:\Omega\to \{0,1\}$, where the two values $0$ and $1$ represent the distinct phases.
The surface area of the interface between both phases describes in many systems a relevant energy, and can be formalized as the \emph{perimeter functional} $\calP:L^1(\Omega,\{0,1\})\to [0,\infty]$,
\begin{equation*}
  \calP(u) = 
  \begin{cases}
    \int_\Omega |\nabla u|\quad&\text{if $u$ is of bounded variation,}\\
    +\infty &\text{ else.}
  \end{cases}
  \label{1eq:Per}
\end{equation*}
The prototype of a mesoscale \emph{diffuse interface} energy is the \emph{Van der Waals--Cahn-Hilliard energy}, given by
\begin{equation}
  \calP_\eps(u) := \int_{\Omega} \Big(\eps |\nabla u|^2 +
  \frac{1}{\eps}W(u)\Big)\,d\calL^n\,.
  \label{1eq:CH}
\end{equation}
Here $\eps>0$ is a small parameter, $W$ is a suitable nonnegative double-well potential with $\{W=0\}=\{0,1\}$ and $u$ is a smooth function on $\Omega\subseteq\R^n$.
To achieve low energy values, the function $u$ has to be close to the wells of the potential except for thin transition layers with thickness of order $\eps$.

The celebrated result by Modica and Mortola \cite{MoMo77,Modi87} describes the sharp interface limit $\eps\to 0$ and connects the mesoscopic diffuse model to the macroscopic sharp interface description:
the functionals $\calP_\eps$ converge in the sense of $\Gamma$-convergence to the perimeter functional,
\begin{equation*}
  \calP_\eps \to 2k \calP\,,\quad k = \int_{0}^1 \sqrt{W}\,.
  \label{eq:c0}
\end{equation*}

Compared to materials science and physics, the mathematical analysis of models for phase separation on biomembranes or artificial membranes made of amphiphilic lipids is much less advanced.
Macroscopic theories for homogeneous membranes focus on \emph{curvature energies} of Canham--Helfrich type \cite{Canh70,Helf73}.
Real-world biomembranes and many artificial membranes on the other hand consist of a number of different types of lipids or other constituents, and conformations are not only determined by the overall shape but also by its internal organization and the partition into distinct phases.

Variational models for multiphase membranes therefore typically include a bending energy with phase-dependent parameters and a line tension energy between phases.
Let us consider the simplest situation of two phases and a membrane represented by a closed surface $S\subset\R^3$ that is decomposed into a disjoint union of open subsets $S_1,S_2$ of $S$ representing the phases, and their common boundary $\Gamma$.
For such configurations extensions of the Canham--Helfrich energies have been proposed by Jülicher--Lipowsky \cite{JuLi96} in the form
\begin{equation}
  \mathcal{E}(S_1,S_2) = \sum_{j=1,2}\int_{S_j} \Big(k_1^j (H-H_0^{j})^2+ k_2^j K\Big)\,d\calH^{2} + \sigma\int_{\Gamma} 1\,d\calH^1\,.
  \label{eq:E-twophase}
\end{equation}
The first integral represents a bending energy that involves in general phase-dependent bending constants $k_1^j,k_2^j$ and the spontaneous curvature $H_0^j$, $j=1,2$.
The second integral in \eqref{eq:E-twophase} describes a phase separation (line tension) energy.
A simple prototype of the bending contribution is the Willmore energy, that is obtained in the case $H_0^j=0=k_2^j$, $j=1,2$.
A rigorous mathematical understanding of Jülicher--Lipowsky-type energies is rather challenging and corresponding results are sparse, see \cite{ChMV13,Helm13,Helm15} for a variational analysis under the assumption of rotational symmetry and \cite{BrLS20} for the general case.
For a recent contribution to sharp interface limits in diffuse phase separation and Jülicher--Lipowsky type energies on generalized hypersurfaces see \cite{OlRoe23}.

\medskip
In this paper we introduce a \emph{mesoscale} model for two-phase membranes that can be expected to carry both an energy contribution due to bending of its overall structure and a phase separation energy.
The model is derived in analogy to (but with a number of substantial reductions) the mesoscale model proposed by Peletier--Röger \cite{PeRoe09}.
For a two-dimensional reduction of that model in \cite{PeRoe09} a rigorous transition to the macroscale has been achieved.
In the limit a generalized Euler elastica functional on systems of $H^2$-regular curves is obtained and it has been shown that the model reflects self-organization in uniformly thin and closed structures, reflecting a strong resistance to fracture and stretching, whereas the bending penalty appears on a smaller energy scale.
Part of the mathematical analysis has been extended to three dimensions \cite{LuPR14,LuRo16}.
Different mesoscale models and their macroscale reductions have been investigated in \cite{SeFr14,Merl15,Merl15a}.

\medskip
For the two-phase mesoscale model introduced in this contribution we present a rigorous analysis in form of a $\Gamma$-convergence result for the phase separation part of the energy that we extract (and generalize) from the full model.
Regarding the full model we have not achieved a full $\Gamma$-convergence result but we present an upper bound estimate with respect to the prospective limit energy, providing the construction of a «recovery sequence».

The analysis for the phase separation energy follows in its main arguments the classical analysis of the \emph{sharp interface limit} of diffuse models of Van der Waals--Cahn-Hilliard type, whereas the construction of recovery sequences  for the full model is analogous to the one in \cite{PeRoe09}.
However, in both parts particular challenges need to be addressed and require some new ideas.
This for example concerns for the phase separation part the number of variables, the scaling of the smooth phase field and additional dependence of the double-well potential on the parameter $\eps>0$.
The main difference in the construction of recovery sequences for the full model is the non-homogeneous thickness of the mesoscale layers, which is enforced by the phase change, and the matching of the mass constraints for both phases.

\medskip
In the next section we briefly discuss the model from \cite{PeRoe09} and introduce the mesoscale model for two-phase membranes made of amphiphilic molecules, in a two-dimensional reduction.
Section \ref{sec:3} presents the analysis of the phase separation part of the full energy, generalized to arbitrary space dimensions.
Results for the full model are given in Section \ref{sec:4}, where we introduce a prospective limit energy and prove the upper bound estimate necessary for a corresponding $\Gamma$-convergence result.
Finally, in Section \ref{sec:5} we summarize our results and discuss some implications and possible extensions.

\section{A mesoscale model for two-phase membranes}\label{sec:2}
Here we start with a mesoscale model for \emph{homogeneous} bilayer membranes as introduced in \cite{PeRoe09}.
That model describes both the self-assembling property of amphiphiles in aqueous solutions and the presence of a resistance against rupture, stretching and bending of the structures.
Moreover, it has very few structural ingredients, notably repulsion between polar and non-polar particles, incompressibility and 
the amphiphilic character of the lipid molecules.
The model in \cite{PeRoe09} considers one type of amphiphilic lipid.
In the mesoscale description its polar «head» and its non-polar «tail» are considered separately and the connection between both is implemented by a «soft constraint».

In this contribution we will address multi-phase configurations, with two types of amphiphilic lipids that are distinguished by their lipid length (in the sense introduced below).

\medskip
For convenience let us first describe the mesoscale model from \cite{PeRoe09}, which was itself derived from a microscopic model, and introduce the changes in the two-phase case below.

\begin{model}[Peletier-Röger \cite{PeRoe09}]
Introduce a scale $\eps>0$ (a posteriori related to the thickness of preferred structures) and a parameter $\Mass>0$ describing a mass constraint.
We then consider a configuration space given by tuples of functions
\begin{align*}
  &\calA_\eps := \Big\{u_\eps\,,v_\eps \in L^1(\R^2,\{0,\eps^{-1}\}) \,:\, \|u_\eps\|_{L^1(\R^n)} =\|v_\eps\|_{L^1(\R^n)} = \Mass \,,
  u_\eps+v_\eps\leq \frac{1}{\eps}\Big\}\,.
\end{align*}
The functions $u_\eps,v_\eps$ describe the distribution of the tail and head particles, respectively.
The constraint of the values to $\{0,\eps^{-1}\}$ and the almost everywhere disjointness of the supports reflects an incompressibility condition and a certain rescaling of the model with $\eps$.

For $(u_\eps\,,v_\eps)\in \calA_\eps$ we introduce the functional
\begin{align}
  \calF_\eps(u_\eps\,,v_\eps) 
  := \eps |\nabla u_\eps|(\R^2) +  \frac{1}{\eps}d_1(u_\eps,v_\eps)\,,
  \label{eq:PRenergy}
\end{align}
where $d_1$ describes the Monge-Kantorovich distance between two mass distributions.

The first term in the energy expresses an repulsion energy between the non-polar tail particles and the polar particles (heads and surrounding water).
The second term is an implementation of the head-tail connection in the lipids in form of a penalization, thus replacing a «hard» by a «soft» constraint.

Exploiting the structure that is induced by the variational characterization of the Monge--Kantorovich distance a (sharp) lower bound for the energy of a configuration is derived in \cite{PeRoe09}.
Therefore a parametrization of the supports of the tail distribution $\{u_\eps=\eps^{-1}\}$ and of the head distribution $\{v_\eps=\eps^{-1}\}$ is used, given by 
\begin{itemize}
  \item a family of arclength-parametrizations $\gamma_i:[0,L_i)\to\R^2$, $i\in I$, of the boundary curves of $\{u_\eps=\eps^{-1}\}$\,, oriented such that $\nu_i:[0,L_i)\to\sphere^1$, $\nu_i=(\gamma_i')^\perp$ is the inner unit normal field of $\{u_\eps=\eps^{-1}\}$ along $\gamma_i$,
  \item parametrizations of \emph{transport rays} with respect to the optimal mass transport associated to the Monge--Kantorovich distance;
  the ray directions induce a \emph{direction field} $\theta_i:[0,L_i]\to\sphere^1$ along $\gamma_i$ such that $\theta_i\cdot\nu_i>0$,
  \item a field $M_i:[0,L_i)\to [0,\infty)$ along $\gamma_i$, where $M_i(s)$ describes the amount of mass being transported through $\gamma_i(s)$ «in the right direction» (see \cite{PeRoe09} for details).
\end{itemize}
This eventually leads to a family of parametrizations
\begin{align*}
  &\psi_i^\eps : \big\{(s,m)\,:\,s\in [0,L_i),\,m\in \big[0,M_i(s)\big)\big\}\to \{u_\eps=\eps^{-1}\}\,,\\
  &\psi_i^\eps(s,m) = \gamma_i(s) + t_i^\eps(s,m)\theta_i(s)\,,
\end{align*}
where
\begin{equation*}
  t_i^\eps(s,m) = \frac{\nu_i(s)\cdot\theta_i(s)}{\theta_i'(s)\cdot\theta_i^\perp(s)}
  \Big[1- \Big(1-\frac{2\theta_i'(s)\cdot\theta_i^\perp(s)\eps} 
  {|\nu_i(s)\cdot\theta_i(s)|^2}m\Big)^\frac{1}{2}\Big]
\end{equation*}
gives the (signed) distance from $\gamma_i(s)$ of a point that corresponds to the relative mass $m$ on the ray through $\gamma_i(s)$; by negative values of $m$ the support of $v_{\eps}$ is parametrized.

This parametrization of the support allows for a lower bound estimate, which by the results in \cite{PeRoe09} is sharp, of the form
\begin{align}
  \calF_\eps(u_\eps,v_\eps) &\geq \sum_{i\in I} \Big(L_i + \int_0^{L_i}\int_0^{M_i(s)}t_i^\eps(s,m)-t_i^\eps(s,m-M_i(s))\,dm\,ds\Big)\nonumber\\
  &\geq \sum_{i\in I} \Big(L_i + \calD_\eps(\gamma_i,\theta_i,M_i)\Big)\,,
  \label{eq:EstCalF}
\end{align}
with a «reduced distance» term
\begin{equation}
  \calD_\eps(\gamma,\theta,M) = \int_0^{L}\Biggl[\frac{1}{\nu(s)\cdot\theta(s)}M(s)^2+
  \frac{\eps^2}{4\big(\nu\cdot\theta\big)^5(s)}|\theta'(s)|^2M(s)^4\Biggr]\,ds\,.
  \label{eq:ModDist}
\end{equation}

If one focusses on configurations with nearly optimal energy in the sense that
\begin{equation*}
  \calF_\eps(u_\eps,v_\eps)-2\Mass\leq \eps^2\Lambda
\end{equation*}
for some $\Lambda>0$ independent of $\eps$, one expects that with $\eps\to 0$ the mass distribution becomes uniformly close to one and that the ray direction is very close to the inner unit normal of $\{u_\eps=\eps^{-1}\}$.
Then the last term in \eqref{eq:ModDist} formally leads with $\eps\to 0$ to a contribution by Eulers elastica energy.

This has been made precise and rigorous in \cite{PeRoe09}, where it is in particular proved that
\begin{equation*}
  \frac{1}{\eps^2}\Big(\calF_\eps(u_\eps,v_\eps)-2\Mass\Big)
\end{equation*}
converge in the sense of $\Gamma$-convergence with $\eps\to 0$ to a multiple of a generalized Eulers elastica energy.
\end{model}

\begin{remark}
Let us remark, that the energy is localized on the different boundary curves $\gamma_i$, and the coupling is only via the total mass constraint and the constraint of non-overlapping supports.
In particular, the key analysis is done in a setting where a tuple $(\gamma,\theta,M)$ and the energy $L(\gamma)+ \calD_\eps(\gamma,\theta,M)$ is evaluated.

\smallskip
Whereas the model from \cite{PeRoe09} is very specific and its motivation from biophysics is rather weak, the reduced energy that appears in \eqref{eq:ModDist} takes a quite common form for energies appearing in biomembrane models.
In fact, the corresponding expression consists of an energy over curves (in other models often a midcurve of the bilayer) with an energy density given by the geometry of the curve and an additional director field.
Therefore, the mathematical analysis and the tools derived in such particular setting also have importance beyond the specific model.
\end{remark}

Let us now derive an analogue model for the case of two different types of amphiphilic lipids.

In contrast to the approach in \cite{PeRoe09} we here start from the reduced description obtained in the lower bound \eqref{eq:EstCalF}.

To motivate the mesoscale energy we first consider a single curve over which head and tail supports are parametrized.

\begin{definition}\label{def:ConfSingleCurve}
Let $\eps>0$ be given.
The \emph{configuration space} $\calZ_\eps$ is defined as the set of tuples $Z=(\gamma,\theta,\chi,M)$ with the following properties:
\begin{enumerate}
  \item\label{it:ZEps-gamma} $\gamma:\R\to\R^2$ is a periodic $C^1$-curve   parametrized by arclength.
  We denote by $L(\gamma)=L$ its minimal period.
  $\theta,\chi,M$ are also $L$-periodic.
  \item\label{it:ZEps-theta} The \emph{ray direction} $\theta:\R\to \sphere^1$ is locally $H^1$-regular.
  \item\label{it:ZEps-M} The \emph{lipid mass distribution} $M:\R\to [0,\infty)$ is locally $H^1$-regular.
  \item\label{it:ZEps-MDisj} The \emph{phase indicator function} $\chi:\R\to\{0,1\}$ is measurable.
  \item\label{it:ZEps-psi} The mapping $\psi=\Psi^\eps(Z)$ defined by
  \begin{align*}
    &\psi : \big\{(s,m)\,:\,s\in [0,L),\,m\in \big(-M(s),M(s)\big)\big\}\to \R^2\,,\\
    &\psi(s,m) = \gamma(s) + t^\eps(s,m)\theta(s)\,,\\
    &t^\eps(s,m) = \frac{\nu(s)\cdot\theta(s)}{\theta'(s)\cdot\theta^\perp(s)}
    \Big[1- \Big(1-\eps\frac{2\theta'(s)\cdot\theta^\perp(s)} 
    {|\nu(s)\cdot\theta(s)|^2}m\Big)^\frac{1}{2}\Big]
  \end{align*}
  is an embedding.
\end{enumerate}
\end{definition}
To $Z=(\gamma,\theta,\chi,M)\in \calZ_\eps$ we can associate the mass distributions $M^{(1)},M^{(2)}$ of the two types of lipids,
\begin{equation}
  M^{(1)}(s) := (1-\chi(s))M(s)\,,\quad
  M^{(2)}(s) := \chi(s)M(s)\,.
  \label{eq:DefMjs}
\end{equation}
These distributions describe how much mass of the respective lipid type is located on the ray through $\gamma(s)$.
We always have $M^{(1)}M^{(2)}=0$, which means that the model forbids any phase change on a single ray. 

The total mass of lipids of type $j\in\{1,2\}$ is given by
\begin{equation}
  M_j(Z)=\int_0^L M^{(j)}(s)\,ds\,.
  \label{eq:DefMj}
\end{equation}

The condition \ref{it:ZEps-psi} in Definition \ref{def:ConfSingleCurve} ensures that $Z$ characterizes well defined tail and head regions of the two types of lipids, given by
\begin{align*}
  \calU^{(j)}_\eps(Z)&=\Big\{\psi(s,m)\,:\, s\in [0,L),\,m\in \big(0,M^{(j)}(s)\big)\Big\}\,,j=1,2\,,\\
  \calV^{(j)}_\eps(Z)&=\Big\{\psi(s,m)\,:\, s\in [0,L),\,m\in \big(-M^{(j)}(s),0\big)\Big\}\,,j=1,2\,.
\end{align*}
  
A key new aspect in the two-phase case is that we need to distinguish the two types of lipids.
We choose here to implement an energy description that induces a different preference for the ray length of type 1 or type 2 lipids.
This is done by choosing different weights for the distance term in the energy functional, expressed by an additional model parameter $\lambda_\eps$ close to one.

\begin{definition}[Multiphase energy for a single curve] \label{def:EnSingleCurve}
Let $\eps>0$, $\sigma>0$ and $\lambda_\eps>0$ be given.
For $Z=(\gamma,\theta,\chi,M)\in\calZ_\eps$ and $M^{(1)},M^{(2)}$ as defined in \eqref{eq:DefMjs} we consider the energy 
\begin{align}
  \tilde F_\eps(Z) = L(\gamma) + \lambda_\eps^2\calD_\eps(\gamma,\theta,M^{(1)})+ \calD_\eps(\gamma,\theta,M^{(2)})+\frac{\sigma}{2} \int_0^L  \eps^2|M'|^2(s)\,ds\,,
  \label{eq:EnSingleCurve}
\end{align}
with $\calD_\eps$ as in \eqref{eq:ModDist}.
\end{definition}
The first three terms in the energy are in analogy with the one-phase bilayer energy from \cite{PeRoe09}.
When looking at configurations made of families of curves and tuples $Z$ as above it turns out that these contributions alone may lead to low energy configurations that are not consistent with the observed structures of phase separated membranes, see Figure \ref{fig:NoLineTension}.
To remedy this issue the fourth contribution in \eqref{eq:EnSingleCurve} is added.
The term may be connected to an additional contribution from the interaction between lipid heads and water molecules.

\begin{remark}[Additional energy term]
Let us give a completely formal indication, under restrictive simplifying assumptions, how to motivate the additional energy term.
We parametrize the boundary curve between lipid head support and the outside water by the mapping $s\mapsto \gamma(s)+t^\eps(s,M(s))\theta(s)$.
Then we define a corresponding \emph{excess energy} by
\begin{equation}
  \int_0^L \sigma\Big[\sqrt{\big(\gamma\cdot \theta^\perp+t^\eps(\cdot,M)\theta'\cdot\theta^\perp)^2 
  +(\gamma'\cdot\theta + t^\eps(\cdot,M)')^2} -1 \Big]\,ds\,.
\label{eq:Motiv}
\end{equation}
The identity for mass coordinates gives
\begin{equation*}
  M(s)= \frac{t^\eps(s,M(s))}{\eps}\nu(s)\cdot\theta(s)-\frac{t^\eps(s,M(s))^2}{2}\theta'(s)\cdot\theta^\perp(s)\,.
\end{equation*}
Differentiating this equations yields
\begin{equation*}
  \eps M' = t^\eps(\cdot,M)'\Big(\nu\cdot\theta-t^\eps(\cdot,M)\theta'\cdot\theta^\perp\Big) 
  +t^\eps(\cdot,M)\big(\nu\cdot\theta\big)'-\frac{t^\eps(s,M(s))^2}{2}\big(\theta'\cdot\theta^\perp\big)'\,.
\end{equation*}
We expect generically $t^\eps(s,M(s))=O(\eps)$, $\theta\cdot\nu = 1+O(\eps)$ and neglect all contributions formally of order $\eps$ or smaller to deduce 
\begin{equation*}
  \eps M' \approx t^\eps(\cdot,M)'\,\,,\quad
  \sqrt{\big(\gamma\cdot \theta^\perp+t^\eps(\cdot,M)\theta'\cdot\theta^\perp)^2 
  +(\gamma'\cdot\theta + t^\eps(\cdot,M)')^2} -1 \approx \frac{1}{2}\eps^2(M')^2\,.
\end{equation*}
Using this in \eqref{eq:Motiv} gives the additional contribution in \eqref{eq:EnSingleCurve}.
\end{remark}
\begin{figure}
\centerline{


\includegraphics[width=0.5\textwidth]{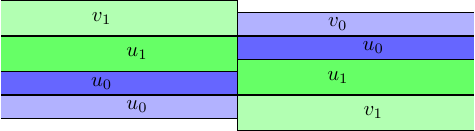}
}
\caption{The total width being equal on either side, switching type does not increase the energy.}
\label{fig:NoLineTension}
\end{figure}

Following \cite{PeRoe09} we can estimate the energy from below by a somewhat simpler expression.
Again this estimate is sharp and it is therefore well motivated to pass to this reduced energy.

\begin{lemma} \label{lem:ModelDer}
Let $Z=(\gamma,\theta,\chi,M)\in\calZ_\eps$ as in Definition \ref{def:ConfSingleCurve} and $M_1(Z), M_2(Z)$ as defined in \eqref{eq:DefMj} be given.
We set $a_\eps(0)=\lambda_\eps$ and $a_\eps(1)=1$.

Then for $\tilde F_\eps$ as in \eqref{eq:EnSingleCurve} the lower estimate
\begin{align}
  \tilde F_\eps(Z) \geq F_\eps(Z) 
  := 2\lambda_\eps M_1(Z) + 2M_2(Z)
  + \eps^2 \big(E_\eps(Z)+G_\eps(Z)\big)
 \label{eq:EnEstSingleCurve}
\end{align}
holds with
\begin{align}
  E_\eps(Z) &= \int_0^L \Big[\frac{1}{\eps^2}(1 - a_\eps(\chi) M)^2 +\frac{\sigma}{2}|M'|^2(s)\Big]\,ds\,,
\label{eq:DefPEps}\\
  G_\eps(Z) &= \int_0^L \Big[\frac{1}{\eps^2}\frac{1-\nu\cdot\theta}{\nu\cdot\theta}a_\eps(\chi)^2 M^2 + \frac{a_\eps(\chi)^2|\theta'|^2M^4}{4(\nu\cdot\theta)^5}\Big]\,ds\,.
\label{eq:DefGEps}
\end{align}
\end{lemma}

\begin{proof}
By completing the squares and using that $M^{(1)}M^{(2)}=0$ for the mass distributions defined in \eqref{eq:DefMjs} we compute
\begin{align*}
  &L(\gamma) + \lambda_\eps^2 \calD_\eps(\gamma,\theta,M_1) + \calD_\eps(\gamma,\theta,M_2) \\
  &\qquad= \int_0^L \Big[1 + \frac{\lambda_\eps^2}{\nu\cdot\theta}(M^{(1)})^2 + \lambda_\eps^2 \frac{\eps^2 |\theta'|^2}{4 (\nu\cdot\theta)^5}(M^{(1)})^4
  + \frac{1}{\nu\cdot\theta}(M^{(2)})^2 + \frac{\eps^2 |\theta'|^2}{4 (\nu\cdot\theta)^5}(M^{(2)})^4\Big]\,ds\\
  &\qquad\geq 2\lambda_\eps M_1(Z_i) + 2M_2(Z_i)\\
  &\qquad\qquad+\int_0^L \Big[(1-a_\eps(\chi) M)^2 +\frac{1-\nu\cdot\theta}{\nu\cdot\theta}a_\eps(\chi)^2 M^2 + \eps^2\frac{a_\eps(\chi)^2|\theta'|^2M^4}{4(\nu\cdot\theta)^5}\Big]\,ds\,.
\end{align*}
Adding the last term in \eqref{eq:EnSingleCurve} and rearranging gives the claim.
\end{proof}

We now fix the general setup and consider a family of curves over which the support of the lipids is parametrized.

\begin{definition}[Mesoscale model]
\label{def:MesoMod}
Let $\eps>0$ and $\Mass_1,\Mass_2>0$ be given.
Denote by $\calA_\eps$ the set of admissable families of tuples $\bldZ=(Z_i)_{i=1,\dots,N}$ such that $N\in\N$,
\begin{align}
  &Z_i\in \calZ_\eps\quad\text{ for all }i=1,\dots,N\,,
  \label{eq:A-Z}\\
  &\big\{\range \Psi^\eps(Z_i)\,:\,i=1,\dots,N\big\}\,\text{ is up to $\calL^2$-nullsets pairwise disjoint,}
  \label{eq:A-UV}\\
  &\sum_{i=1}^N M_j(Z_i) = \Mass_j\quad\text{ for }j=1,2\,.
 \label{eq:A-M}
\end{align}
For $\bldZ\in \calA_\eps$ we define
\begin{equation}
  \calF_\eps(\bldZ)= \sum_{i=1}^N F_\eps(Z_i)\,,
  \label{eq:MesoEnergy}
\end{equation}
with $F_\eps$ as defined in \eqref{eq:EnEstSingleCurve}.
\end{definition}

Lemma \ref{lem:ModelDer} yields that
\begin{equation*}
  \calF_\eps(\bldZ) = 2\lambda_\eps \Mass_1 + 2\Mass_2
  + \eps^2 \big(\calP_\eps(\bldZ)+\calG_\eps(\bldZ)\big)\,,
\end{equation*}
where
\begin{equation}
  \calP_\eps(\bldZ)= \sum_{i=1}^N P_\eps(Z_i)\,,\quad
  \calG_\eps(\bldZ)= \sum_{i=1}^N G_\eps(Z_i)\,.
  \label{eq:DefCalPCalG}
\end{equation}

We expect that $\calP_\eps$ describes a phase separation energy and that $\calG_\eps$ describes a bending energy.
This in particular requires a suitable choice of the dependence of $\lambda_\eps$ on $\eps$.

The main goal of this paper is to make the first claim rigorous by proving an appropriate $\Gamma$-convergence result for the mesoscale phase separation energy $\calP_\eps$.

Regarding the bending contribution we will provide the construction of a recovery sequence, that is the upper bound estimate in a corresponding $\Gamma$-convergence result.

\section{\texorpdfstring{$\Gamma$-convergence of a mesoscale phase separation energy}{Gamma-convergence of a mesoscale phase separation energy}}\label{sec:3}

We start with fixing the scaling behavior of the parameter $\lambda_\eps$ with $\eps$.
Fix a constant $c\geq 0$ and let
\begin{equation}
  \lambda_\eps = \frac{1}{1-c\sqrt\eps}\,.
\label{eq:DefLambdaEps}
\end{equation}
The analysis below shows that exactly for this scaling the contributions from bending and the phase separation appear on the same order with respect to the thickness of low energy structures.

In order to prepare extensions of the analysis we generalize the phase separation energy in the mesoscale model to a higher-dimensional setting.
For simplicity we choose $\sigma=1$.

\begin{definition}[Phase separation energies]
\label{def:PhaseSepEn}
Let $\Omega\subset \R^d$ be a bounded domain with $C^1$-regular boundary.
For $\eps>0$ consider $\lambda_\eps$ as in \eqref{eq:DefLambdaEps}, set  $a_\eps(0)=\lambda_\eps$, $a_\eps(1)=1$ and define the mesoscale energies $E_\eps: L^1(\Omega)\times L^1(\Omega)\to [0,\infty]$ by
\begin{align*}
  E_{\eps}(\me,\ch) \coloneqq 
  \begin{cases}
    \int_{\Omega} \frac{1}{\eps^2} \left(1 - a_\eps(\chi_{\eps}) M_{\eps} \right)^2 + \frac{1}{2}|\nabla M_{\eps}|^2 \,dx
    &\text{ if }(\me,\ch)\in H^1(\Omega)\times L^1(\Omega,\{0,1\})\,,\\
    +\infty &\text{ else.}
  \end{cases}
\end{align*}
In addition, define a macroscale energy $E: L^1(\Omega)\times L^1(\Omega)\to [0,\infty]$ by
\begin{align*}
  E(M,\chi) \coloneqq 
  \begin{cases}
    \frac{c^2}{\sqrt{8}} |\nabla \chi|(\Omega) 
    &\text{if } M=1\text{ and } \chi \in \BV(\Omega,\{0,1\})\,, \\
    +\infty & \text{else.}
  \end{cases}
\end{align*} 
\end{definition}

We next formulate the main result of this section.
\begin{theorem}[$\Gamma$-convergence]
\label{simp:gamma}
The family $(E_{\eps})_{\eps>0}$ $\Gamma$-converges to $E$ in $L^1(\Omega)\times L^1(\Omega)$.

More precisely, the following holds:
\begin{enumerate}
\item \emph{Compactness.} 
For any sequence $\epsk\to 0$ ($k\to\infty$) and any sequence  $(M_k,\chi_k)_{k\in\N}$ in $L^1(\Omega)\times L^1(\Omega)$ with
\begin{equation}
  \liminf_{k\to\infty} E_{\epsk}(M_k,\chi_k) < \infty
\label{eq:LiminfBound}
\end{equation}
there exists a subsequence $k\to\infty$ (not relabeled) and a function $\chi\in \BV(\Omega,\{0,1\})$ such that
\begin{equation}
  M_k\to 1\,, \chi_k\to \chi\quad\text{ in }L^1(\Omega)\,.
\label{eq:Cpct}
\end{equation}
\item \emph{Lower bound.}
For any sequence $\epsk\to 0$ ($k\to\infty$) and any sequence $(M_k,\chi_k)_k$ in $L^1(\Omega)\times L^1(\Omega)$ with $(M_k,\chi_k)_k\to (M,\chi)$ in $L^1(\Omega)\times L^1(\Omega)$ we have
\begin{equation}
  E(M,\chi) \leq \liminf_{k\to\infty} E_{\epsk}(M_k,\chi_k)\,.
\label{eq:Liminf}
\end{equation}
In particular, $M=1$ and $\chi\in \BV(\Omega,\{0,1\})$ holds if the right-hand side is finite.
\item \emph{Upper bound.}
For any $\chi\in \BV(\Omega,\{0,1\})$ there exist $\epsk\to 0$ ($k\to\infty$) and a sequence $(M_k,\chi_k)_k$ in $L^1(\Omega)\times L^1(\Omega)$ with $(M_k,\chi_k)_k\to (1,\chi)$ in $L^1(\Omega)\times L^1(\Omega)$ such that
\begin{equation}
  E(1,\chi) \geq \limsup_{k\to\infty} E_{\epsk}(M_k,\chi_k)\,.
\label{eq:Limsup}
\end{equation}
\end{enumerate}
\end{theorem}

\begin{proof}
We split the proof into several steps.
For convenience we set $\lambda_k=\lambda_{\epsk}$ and $a_k=a_{\epsk}$.
\begin{enumerate}[wide,label=(\roman*)]
\item \emph{Setup.}
Consider any sequences $\epsk\to 0$ ($k\to\infty$),  $(M_k,\chi_k)_{k\in\N}$ in $L^1(\Omega)\times L^1(\Omega)$ with \eqref{eq:LiminfBound}.
Without loss of generality we can assume $\epsk\leq 1$ for all $k\in\N$ and
\begin{equation}
  E_{\epsk}(M_k,\chi_k) < \Lambda\quad\text{ for all }k\in\N\,.
\label{eq:BoundLambda}
\end{equation}
Moreover, passing possibly to a subsequence of $k\to\infty$, we can assume that the limit
\begin{equation}
  \lim_{k\to\infty} E_{\epsk}(M_k,\chi_k) =\liminf_{k\to\infty} E_{\epsk}(M_k,\chi_k)\leq\Lambda
\label{eq:LimAss}
\end{equation}
exists.
\item \emph{Convergence of mass distributions.}
We begin by proving that
\begin{equation}
  M_k\to 1 \quad\text{ in }L^1(\Omega)\,.
\label{eq:ConvMk}
\end{equation}
By the boundedness of the energy we obtain that
\begin{align*}
  \Lambda \epsk^2 \geq \int_{\Omega} \left(1 - a_k(\chi_k)M_k \right)^2\,.
\end{align*}
Since $a_k$ converges to $1$ uniformly we deduce that in $L^2(\Omega)$ we have
\begin{equation*}
  \lim_{k\to\infty} M_k 
  = \lim_{k\to\infty} \frac{1}{a_k(\chi_k)}\Big(1-\big(1-a_k(\chi_k)M_k\big)\Big) 
  =1\,,
\end{equation*}
Since $\Omega$ is bounded this implies \eqref{eq:ConvMk}.

\item We next provide an estimate from below and a reduction to a one-variable energy.
We therefore associate to  $M_k$ an optimal phase distribution $\och$ by introducing the threshold
\begin{align}
  a^*_k \coloneqq 1 - \frac{\cek}{2-\cek} = \frac{2}{\lambda_k+1}
\label{eq:DefAStar}
\end{align}
and letting
\begin{align*}
  \och \coloneqq \chi_{\{ M_k > a^*_k \}}\,.
\end{align*}
We then define the reduced energy $\overline{E}_k:L^1(\Omega)\to [0,\infty]$,
\begin{align*}
  \overline{E}_k(M) \coloneqq E_{\epsk}(M,\chi_{\{ M > a^*_k \}})
\end{align*}
and claim that
\begin{equation}
  \overline{E}_k(M_k) \leq E_{\epsk}(M_k,\chi_k)
\label{eq:EstBarEk}
\end{equation}
holds.
In fact, we have
\begin{align*}
  &(1-a_k(0)M_k)^2 - (1-a_k(1)M_k)^2 \nonumber\\
  &\qquad = (1-\lambda_kM_k)^2 - (1-M_k)^2 
  = M_k(\lambda_k-1)\Big((\lambda_k+1)M_k-2\Big)
  \leq 0
\end{align*}
if and only if $M_k\leq a_k^*$, and
\begin{equation}
  (1-a_k(0)a_k^*)^2 = (1-a_k(1)a_k^*)^2\,.
\label{eq:AStarProp1}
\end{equation}
Hence $(1-a_k(\och)M_k)^2\leq  (1-a_k(\chi_k)M_k)^2$ and \eqref{eq:EstBarEk} follows.
We will prove later that $\chi_k$ and $\och$ are very close. 

\item We next provide a lower bound for the single-variable energy $\overline{E}_k$ by the usual «Modica--Mortola trick».
Young's inequality yields for any $k\in\N$
\begin{align}
  \overline{E}_k(M_k)&=\int_{\Omega} \left[ \frac{1}{\epsk^2} \left(1 - a_k(\och(M_k)) M_k \right)^2 + \frac{|\nabla M_k|^2}{2} \right]\,dx \nonumber\\
  &\geq \int_{\Omega} \frac{\sqrt{2}}{\epsk} \big| \left( 1 - a_k(\och(M_k)) M_k \right) \nabla M_k \big|\,dx 
  =\int_{\Omega} |\nabla H_k(M_k)|\,dx\,,
\label{eq:MoMoTrick}
\end{align}
where we have chosen $H_k$ as
\begin{align*}
  H_k(t) = \int_{1-c\sqrt{\epsk}}^{t} \frac{\sqrt{2}}{\epsk} \left|1 - a_k(\och(s)) s \right| \,ds\,.
\end{align*}
\item  We collect some properties of the functions $H_k$.
These functions are monotonically increasing with
\begin{equation*}
  H_k(M)=0 \quad\text{ for }M=\frac{1}{a_k(0)}=\frac{1}{\lambda_k}= 1 - c\sqrt{\epsk}\,.
\end{equation*}
By using \eqref{eq:AStarProp1} we obtain
\begin{align}
  \frac{1}{\sqrt{2}\epsk\lambda_k}(\lambda_ka_k^*-1)^2 
  + \frac{1}{\sqrt{2}\epsk}(1-a_k^*)^2= \frac{1}{\sqrt{2}\epsk}\frac{(\lambda_k-1)^2}{\lambda_k(\lambda_k+1)} = \frac{c^2}{\sqrt{2}(2-c\sqrt{\epsk})}\,.
\label{eq:Hk1a}
\end{align}
We deduce that
\begin{align}
  H_k(t)= 
  \begin{cases}
    -\frac{1}{\sqrt{2}\lambda_k\epsk}(1-\lambda_kt)^2
    &\text{ for }t<\frac{1}{\lambda_k}\,,\\
    \frac{1}{\sqrt{2}\lambda_k\epsk}(1-\lambda_kt)^2
    &\text{ for }\frac{1}{\lambda_k}\leq t \leq a_k^*\,,\\
    \frac{c^2}{\sqrt{2}(2-c\sqrt{\epsk})}-\frac{1}{\sqrt{2}\epsk}(1-t)^2
    &\text{ for }a_k^*\leq t\leq 1\,,\\
    \frac{c^2}{\sqrt{2}(2-c\sqrt{\epsk})}+\frac{1}{\sqrt{2}\epsk}(t-1)^2
    &\text{ for }t> 1\,.
  \end{cases}
\label{eq:Hk1b}
\end{align}
In particular,
\begin{align}
  H_k(1)&= \frac{c^2}{\sqrt{2}(2-c\sqrt{\epsk})}\,.
\label{eq:Hk1}
\end{align}
Moreover, there exists $C>0$ such that
\begin{equation}
  |H_k(t)| \leq C+ \frac{1}{\epsk^2}(1-a_k(t)t)^2
  \quad\text{  for all }k\in\N\,,t\in\R\,.
\label{eq:EstHk}
\end{equation}
\item We next provide a uniform bound for $(H_k(M_k))_{k}$ in $W^{1,1}(\Omega)$.
We first deduce from \eqref{eq:BoundLambda}, \eqref{eq:EstBarEk} and \eqref{eq:MoMoTrick} that
\begin{equation}
  \int_{\Omega} |\nabla H_k(M_k)|\,dx \leq \Lambda\quad\text{ for all }k\in\N\,.
\label{eq:BVBound}
\end{equation}
Furthermore, by \eqref{eq:BoundLambda}, \eqref{eq:EstBarEk} and \eqref{eq:EstHk} we have
\begin{equation}
  \int_{\Omega} |H_k(M_k)|\,dx \leq C(1+\Lambda)\quad\text{ for all }k\in\N\,.
\label{eq:L1Bound}
\end{equation}
\item \emph{Compactness of $(H_k(M_k))_{k}$ in $L^1(\Omega)$.}
By the Rellich theorem and \eqref{eq:BVBound}, \eqref{eq:L1Bound} there exists a subsequence of $k\to\infty$ (not relabeled) and function $u\in L^1(\Omega)$ such that
\begin{equation}
  H_k(M_k) \to u \quad\text{ in }L^1(\Omega)\,.
\label{eq:ConvHkMk}
\end{equation}
From the lower-semicontinuity property of the total variation and \eqref{eq:LimAss}, \eqref{eq:BoundLambda}, \eqref{eq:EstBarEk} and \eqref{eq:MoMoTrick} we further deduce $u\in\BV(\Omega)$ with
\begin{equation}
  \int_{\Omega} |\nabla u| \leq \liminf_{k\to\infty}\int_{\Omega} |\nabla H_k(M_k)|\,dx \leq \liminf_{k\to\infty} E_{\epsk}(M_k,\chi_k)\,.
\label{eq:LimInf-1}
\end{equation}
\item \emph{Compactness of $(\och)_{k}$ in $L^1(\Omega)$.}
We claim that
\begin{equation}
  \och \to \chi:= \frac{\sqrt{8}}{c^2}u \quad\text{ in }L^1(\Omega)\,.
\label{eq:ConvBarChi}
\end{equation}
In fact, we compute
\begin{align*}
  \Big|H_k(M_k)-\frac{c^2}{\sqrt{8}}\och\Big|
  &=(1-\och)|H_k(M_k)| +\och\Big|H_k(M_k)-\frac{c^2}{\sqrt{8}}\Big|\\
  &\leq (1-\och)\frac{1}{\sqrt{2}\lambda_k\epsk}(1-\lambda_kM_k)^2
  + \och\Big(\Big|H_k(1)-\frac{c^2}{\sqrt{8}}\Big| + \frac{1}{\sqrt{2}\epsk}(1-M_k)^2\Big)\\
  &\leq C\Big(\sqrt{\epsk} + \epsk\frac{1}{\epsk^2}\big(1-a_k(\och)M_k\big)^2 \Big)\,.
\end{align*}
Therefore,
\begin{align*}
  \Big\|H_k(M_k)-\frac{c^2}{\sqrt{8}}\och\Big\|_{L^1(\Omega)} 
  \leq C(\Omega)\Big(\sqrt{\epsk} + \epsk\overline{E}_k(M_k)\Big)\to 0 \,(k\to\infty)\,,
\end{align*}
and by \eqref{eq:ConvHkMk} we conclude \eqref{eq:ConvBarChi}.
Since the functions $\och$ are $\{0,1\}$-valued we deduce that $\chi\in L^1(\Omega,\{0,1\})$.
Furthermore, since $\chi= \frac{\sqrt{8}}{c^2}u$ we have $\chi\in \BV(\Omega,\{0,1\})$ and by \eqref{eq:LimInf-1}
\begin{equation}
  \frac{c^2}{\sqrt{8}}\int_{\Omega} |\nabla \chi| \leq \liminf_{k\to\infty}\int_{\Omega} |\nabla H_k(M_k)|\,dx \leq \liminf_{k\to\infty} E_{\epsk}(M_k,\chi_k)\,.
  \label{eq:LimInf-2}
\end{equation}
\item \emph{Compactness and lower semicontinuity.}
We first show that $\chi$ is also the $L^1(\Omega)$-limit of $\chi_k$.
In fact we deduce from \eqref{eq:BoundLambda} that
\begin{align*}
  \Lambda &> E_{\epsk}(M_k,\chi_k) 
  \geq \int_{\Omega} \frac{1}{\epsk^2} \big(1 - a_k(\chi_k)M_k\big)^2 \,dx\\
  &\geq \frac{1}{\epsk^2} \left[ \int_{\{M_k > a_k^*\} \cap \{\chi_k=0\}} ( 1 - \lambda_{k} M_k )^2 \,dx + \int_{\{M_k \leq a_k^*\} \cap \{\chi_k=1\}} (1-M_k)^2 \,dx \right]\\
  &\geq \frac{1}{\epsk^2} \left[ ( 1 - a_k^*\lambda_k)^2 |\{\och-\chi_k=1\}| + ( 1-a_k^*)^2 |\{\och-\chi_k=-1\}| \right]\,.
\end{align*}
By \eqref{eq:DefAStar}, \eqref{eq:AStarProp1} it holds
\begin{equation*}
  ( 1 - a_k^*\lambda_k)^2=(1-a_k^*)^2
  =\frac{c^2 \epsk}{(2 - c\sqrt{\epsk})^2}\,,
\end{equation*}
hence we deduce that for some $C>0$
\begin{equation*}
  \|\och-\chi_k\|_{L^1(\Omega)} \leq C\epsk\Lambda\,,
\end{equation*}
which by \eqref{eq:ConvBarChi} implies
\begin{equation*}
  \chi_k\to \chi \quad\text{ in }L^1(\Omega)\,.
\end{equation*}
Together with \eqref{eq:ConvMk} this proves the compactness claim in Theorem \ref{simp:gamma}.

The lower bound statement has already been proved in \eqref{eq:LimInf-2}.

\item \emph{Optimal profile.}
To prepare the construction of a recovery sequence we would like to construct a suitable one-dimensional transition profile between the preferred states $M_k=\frac{1}{\lambda_k}<1$ and $M_k=1$ for the energy $\overline{E}_k(M_k)$.

We therefore recall the application of Young's inequality in \eqref{eq:MoMoTrick}. 
A minimizer $M_k$ of $\overline{E}_{\epsk}$ should achieve equality in this estimate. 
Considering the one-dimensional case this motivates to look for a function $q:\R\to\R$ such that
\begin{align*}
  q' = \frac{\sqrt{2}}{\epsk} \Big| 1-a_k(\bar\chi_k(q))q\Big| \,,\quad
  \bar\chi_k(q)=\chi_{(a_k^*,\infty)}(q)\,,
\end{align*}
subject to the far field conditions and a normalization,
\begin{equation*}
  \lim_{r\to -\infty}q(r)=\frac{1}{\lambda_k}\,,\quad
  \lim_{r\to +\infty}q(r)=1\,,\quad
  q(0)=a_k^*\,.
\end{equation*}
We obtain the solution
\begin{align}
  q_{\epsk}(r) = \left\{ 
  \begin{array}{rcl} 
    (1-\cek) + \frac{(1-\cek)\cek}{2-\cek} \exp \left( \frac{\sqrt{2} r}{\epsk(1-\cek)} \right)\,, & r\leq 0\,,\\
    1 - \frac{\cek}{2-\cek}\exp\left( -\frac{\sqrt{2}r}{\epsk}\right)\,, & r\geq 0\,,
  \end{array} \right.
\label{eq:OptProfile}
\end{align}
and 
\begin{align}
  \frac{\sqrt{2}}{\epsk}\Big|1-a_k(\och(q_{\epsk}(r)))q_{\epsk}(r)\Big| 
  = q_{\epsk}'(r) = \left\{ 
  \begin{array}{rcl} 
    \frac{c\sqrt{2}}{\sqrt{\epsk}(2-\cek)}\exp \left( \frac{\sqrt{2} r}{\epsk(1-\cek)} \right)\,, & r\leq 0\,,\\
    \frac{c\sqrt{2}}{\sqrt{\epsk}(2-\cek)}\exp\left( -\frac{\sqrt{2}r}{\epsk}\right)\,, & r\geq 0\,,
  \end{array} \right.
\label{eq:Equipart}
\end{align}
Note that, in contrast to the classical case, $q_{\epsk}$ cannot be represented as a simple rescaling of an $k$-independent profile.

\item \emph{Construction of a recovery sequence.}
Let $\chi \in \BV(\Omega,\{0,1\})$ be given. We choose $\chi_k = \chi$ for all $k\in\N$.
To construct suitable functions $M_k$ we first consider the case where $\{\chi=1\} = V \cap \Omega$, for $V \subset \R^d$ open and $\partial V$ a non-empty compact $C^2$-hypersurface with $\calH^{d-1}(\partial V \cap \partial \Omega)=0$. Then there is a $\delta>0$ such that the signed distance function $\sdist(\cdot,\partial V)$, taken positive inside $V$, is smooth on $V_{\delta} \coloneqq \{x \in \Omega : |\sdist(\cdot,\partial V)|<\delta\}$ with $|\nabla \sdist(\cdot,\partial V)|=1$. We define
\begin{align*}
M_k(x) \coloneqq q_{\epsk}(\sdist(x,\partial V)).
\end{align*}
Note that, even without any assumptions on $V$, $M_k$ is Lipschitz as a composition of Lipschitz functions. Since $\Omega$ is a bounded domain with $C^1$-regular boundary we apply \cite[5.8, Thm.~4]{Ev10} and Hölder's inequality to obtain that $M_k \in C^{0,1}(\Omega)=W^{1,\infty}(\Omega) \subset H^1(\Omega)$.
By normalization and monotonicity of $q_{\epsk}$ we have then that
\begin{align*}
\och = \chi_k = \chi \quad \text{for all $k \in \N$.}
\end{align*}
We next calculate the energy $E_{\epsk}(M_k,\chi_k)$. 
Firstly we obtain in 
\begin{align*}
  V_{\delta}^+ \coloneqq \{x\in\Omega\,:\, 0<\sdist(x,\partial V)<\delta\} \subset \{\chi = 1\}
\end{align*}
with \eqref{eq:Equipart}
\begin{align*}
  \frac{1}{\epsk^2}(1-M_k)^2 = \frac{c^2}{\epsk(2-\cek)^2}\exp\left( -\frac{ \sqrt{8} \sdist(x,\partial V)}{\epsk} \right) = \frac{|\nabla M_k|^2}{2}
\end{align*}
and similarly in $V_{\delta}^-:=\{x\in\Omega|0<-\sdist(x,\partial V)<\delta\}\subset \{\chi=0\}$
\begin{align*}
  \frac{1}{\epsk^2} \left( 1 - \frac{M_k}{1-\cek} \right)^2 = \frac{c^2}{\epsk(2-\cek)^2}\exp\left( \frac{\sqrt{8} \sdist(x,\partial V)}{\epsk(1-\cek)} \right) = \frac{|\nabla M_k|^2}{2}.
\end{align*}
To estimate the contribution from $V_{\delta}^+$ to $E_{\epsk}(M_k,\chi_k)$ we apply the coarea formula for Lipschitz functions, using that $|\nabla \sdist(\cdot,\partial V)|=1$:
\begin{align*}
  &\frac{2c^2}{\epsk(2-\cek)^2} \int_{V_{\delta}^+} \exp \left( - \frac{\sqrt{8}\sdist(x,\partial V)}{\epsk} \right)\,dx\\
  = &\frac{2c^2}{\epsk(2-\cek)^2} \int_0^{\delta} \exp \left( - \frac{\sqrt{8}r}{\epsk} \right) \calH^{d-1}\big(\{x\in\Omega\,:\,\sdist(x,\partial V)=r\}\big)\,dr\\
  \leq &\sup_{t \in (-\delta,\delta)}\left( \calH^{d-1}\big( \{x\in\Omega\,:\,\sdist(x,\partial V)=t\} \big) \right) \frac{c^2}{\sqrt{2}(2-\cek)^2}\Big( 1 - \exp \Big( \frac{-\sqrt{8}\delta}{\epsk} \Big) \Big)\,.
\end{align*}
Similarly one obtains the contribution from $V_{\delta}^-$
\begin{align*}
  &\frac{2c^2}{\epsk(2-\cek)^2} \int_{V_{\delta}^-} \exp \left( \frac{\sqrt{8}\sdist(x,\partial V)}{\epsk(1-\cek)} \right)\,dx\\
  = &\frac{2c^2}{\epsk(2-\cek)^2} \int_{-\delta}^0 \exp \left( \frac{\sqrt{8}r}{\epsk(1-\cek)} \right) \calH^{d-1}\big(\{x\in\Omega\,:\,\sdist(x,\partial V)=r\}\big)\,dr\\
  \leq &\sup_{t \in (-\delta,\delta)}\left( \calH^{d-1}\big( \{x\in\Omega\,:\,\sdist(x,\partial V)=t\} \big) \right) \frac{c^2(1-\cek)}{\sqrt{2}(2-\cek)^2}\Big( 1 - \exp \Big( \frac{-\sqrt{8}\delta}{\epsk(1-\cek)} \Big) \Big)\,.
\end{align*}
Using the regularity assumptions on $V$ we obtain from \cite[Lemma 5.8]{Leon13} that for $\delta \rightarrow 0$
\begin{align*}
  \sup_{t \in (-\delta,\delta)}\left( \calH^{d-1}\big( \{x\in\Omega\,:\,\sdist(x,\partial V)=t\} \big) \right) \rightarrow  \calH^{d-1}(\Omega \cap \partial V) = |\nabla \chi|(\Omega).
\end{align*}
Taking first the limit in $k \rightarrow \infty$ and then the limit in $\delta \rightarrow 0$ we obtain that the energy contribution of $V_{\delta}^+ \cup V_{\delta}^-$ is bounded in the limit by
\begin{align*}
\frac{c^2}{\sqrt{8}} |\nabla\chi|(\Omega).
\end{align*}
It remains to show that the contribution from $\Omega \setminus \big(V_{\delta}^+ \cup V_{\delta}^-)$ is negligible. 
We write the set as the disjoint union
\begin{align*}
  \Omega \setminus \big(V_{\delta}^+ \cup V_{\delta}^-) 
  = \big( (\Omega \cap V) \setminus V_{\delta}^+ \big) \cup \big( (\Omega\cap V^c) \setminus V_{\delta}^- \big)\,,
\end{align*}
and omit the nullset $\partial V\cap\Omega$. 
Let $x\in (\Omega \cap V) \setminus V_{\delta}^+$, then $\chi(x)=1$ and $\sdist(x,\partial V) \geq \delta$. 
Since $(1-q_{\epsk})$ is positive and decreasing monotonically on $\R_+$ we obtain by using \eqref{eq:OptProfile} in the equality below
\begin{align*}
  \int_{V \setminus V_{\delta}^+} \frac{1}{\epsk^2}(1-M_k(x))^2 \,dx &\leq \calL^d(\Omega) \Big( \frac{1-q_{\epsk}(\delta)}{\epsk} \Big)^2\\
  &= \frac{c^2\calL^d(\Omega)}{\epsk (2-\cek)^2} \exp \Big( - \frac{\sqrt{8}\delta}{\epsk} \Big) \rightarrow 0 \,(k\to\infty)\,.
\end{align*}
Restricted to $\R_+$ the function $q_{\epsk}$ is smooth and $q_{\epsk}'$ is positive and monotonically decreasing. 
Noting again that $\sdist(\cdot,\partial V) \in C^{0,1}(\Omega) = W^{1,\infty}(\Omega) \subset H^1(\Omega)$ with $|\partial_j \sdist(\cdot,\partial V)|\leq 1$ (see for example the proof of \cite[Theorem 10.5]{Alt16}) we obtain
\begin{align*}
  \int_{V \setminus V_{\delta}^+} \frac{|\nabla M_k(x)|^2}{2} \,dx 
  &= \frac{1}{2} \int_{V \setminus V_{\delta}^+} (q_{\epsk}'(\sdist(x,\partial V)))^2 |\nabla \sdist(x,\partial V)|^2\,dx\\
  &\leq \frac{1}{2}\big(q_{\epsk}'(\delta)\big)^2 n|\Omega|\\
  &= \frac{c^2 n|\Omega|}{\epsk (2-\cek)^2} \exp \Big( -\frac{\sqrt{8}\delta}{\epsk} \Big) \rightarrow 0 \,(k\to\infty)\,.
\end{align*}
Thus the energy contribution of $((\Omega \cap V) \setminus V_{\delta}^+)$ is negligible.
By completely analogous arguments the same applies to the contribution from $( (\Omega \cap V^c) \setminus V_{\delta}^-)$ to the energy.

Lastly we extend the result to general functions $\chi \in \BV(\Omega,\{0,1\})$. 
By \cite[Lemma 1.15]{Leon13} there are open sets $V_k \subset \R^d$, where $\partial V_k$ are non-empty compact $C^2$-hypersurfaces and $\calH^{d-1}(\partial \Omega \cap \partial V_k)=0$, such that $\chi_{V_k}|_{\Omega} \rightarrow \chi$ in $L^1(\Omega)$ and $|\nabla \chi_{V_k}|(\Omega) \rightarrow |\nabla \chi|(\Omega)$. 
Constructing $M_k$ depending on $\chi_{V_k}$ as above, we obtain the recovery sequence $(M_k,\chi_{V_k})$ for $(1,\chi)$ with the required properties.
\end{enumerate}
\end{proof} 

\section{Single curve two-phase energy: recovery sequence construction}\label{sec:4}
Here we consider the full two-phase mesoscale energy with both a bending energy and a phase separation contribution.
We restrict ourselves to the two-dimensional case and the simplified setting of just one mono-layer.

We aim at justifying that the mesoscale energy in fact approximates a macroscopic two-phase bending energy with line tension.

On the mesoscale we consider $\eps>0$ and tuples $Z=(\gamma,\theta,\chi,M)\in\calZ_{\eps}$ as in Definition \ref{def:EnSingleCurve}.
We consider the rescaled reduced energy that we obtained in Lemma \ref{lem:ModelDer},
\begin{equation}
  \calE_\eps(Z) := E_\eps(Z) + G_\eps(Z)\,,
\label{eq:DefEeps}
\end{equation}
with $E_\eps,G_\eps$ as defined in \eqref{eq:DefPEps}, \eqref{eq:DefGEps},
and with the choice \eqref{eq:DefLambdaEps} for the parameter $\lambda_\eps$.

We next introduce the macroscale energy.
\begin{definition}[2D, 2-phase, single-curve Jülicher-Lipowsky energy]
\label{def:JL-2D}
Consider a tuple $(\gamma,\chi)$ of a closed curve $\gamma\in H^2_{\loc}(\R)$ parametrized by arclength with minimal period $L$ such that $\gamma|_{[0,L)}$ is an embedding, and an $L$-periodic function $\chi\in \BV_{\loc}(\R)$.
Let $M_1:= \int_0^L\chi\,ds$ and $M_2:= L-M_1$ and define
\begin{equation}
  \calE(\gamma,\chi) = \int_0^L \frac{1}{2}\kappa^2\,ds + \frac{c^2}{\sqrt{8}}|\nabla\chi|\big([0,L)\big)\,.
\label{eq:ELimsup}
\end{equation}
\end{definition}

\begin{theorem}
Consider any $(\gamma,\chi)$ as in Definition \ref{def:JL-2D} with $\calE(\gamma,\chi)<\infty$.
Then there exists a sequence $\epsk\to 0$ ($k\to\infty$) and a sequence of tuples $(Z_k)_k$, $Z_k=(\gamma_k,\theta_k,\chi_k,M_k)\in\calZ_{\epsk}$ as in Definition \ref{def:ConfSingleCurve} such that
\begin{align}
  &M_1(Z_k)=M_1\,,\quad M_2(Z_k)=M_2\quad\text{ for all }k\in\N\,,
\label{eq:MassConstraint}\\
  &(\gamma_k,\theta_k) \weakto (\gamma,\nu) \quad\text{ in }H^1_{\loc}(\R)^2\,,
\label{eq:GammaThetaConv}\\
  &(\chi_k,M_k) \to (\chi,1) \quad\text{ in }L^1_{\loc}(\R)^2\,,
\label{eq:ChiMConv}
\end{align}
and such that
\begin{equation}
  \calE_{\epsk}(Z_k)\to \calE(\gamma,\chi)\quad\text{ as }k\to\infty\,.
\label{eq:UpperBdFull}
\end{equation}  
\end{theorem}

\begin{proof}
We mainly follow the construction already used in \cite{PeRoe09}, using instead of a constant mass distribution the optimal profile determined in the proof of Theorem \ref{simp:gamma}.
The main additional difficulty is to maintain the prescribed total mass for the separate phases \eqref{eq:MassConstraint}.

\begin{enumerate}[wide,label=(\roman*)]
\item  We first consider the case that $\gamma$ is in addition to the properties prescribed in Definition \ref{def:JL-2D} smooth with $|\{\gamma''=0\}|=0$. 
Since $|\nabla\chi|\big([0,L)\big)<\infty$ we can choose $\chi$ as the absolutely continuos representation, which has a countable number of jump points $J\subset \R$ that describe the support of the measure $|\nabla\chi|$.
Then $J\cap [0,L)$ is finite and without loss of generality we can assume that $0\not\in J$, which means that $\chi$ is locally constant at $0$.
Moreover, choose $\eps_0>0$ such that $3\eps_0<\min\{|x-y|\,:\,x\neq y\,, x,y\in J\}$ and denote by $\sdist(\cdot,J):\R\to\R$ the signed distance from $J$, taken positive inside $\{\chi=1\}$.

Choose a smooth unit normal field along $\gamma$ by letting $\nu \coloneqq (\gamma')^{\perp}$ and set $\kappa:=-\gamma''\cdot \nu$.
Since $|\gamma'|=1$ we have $\gamma''\cdot\gamma'=0$ and deduce
\begin{equation}
  |\{\kappa=0\}|=0\,.
\label{eq:Kappa}
\end{equation}

\item In order to satisfy the mass constraint we use a construction that is based on the implicit function theorem.
We consider two given smooth $L$-periodic fields $\rho,\sigma:\R\to\R^2$ that we determine below.
For $r,t\in (-\delta_0,\delta_0)$ with $\delta_0>0$ sufficiently small we define the perturbed curves
\begin{align*}
  \tilde{\gamma}_{r,t} \coloneqq \gamma + \left( r\rho + t\sigma \right) \nu\,,
\end{align*}
and the arc-length parametrization
\begin{align*}
  \gamma_{r,t} \coloneqq \tilde{\gamma}_{r,t} \circ \psi_{r,t}^{-1}\,,
\end{align*}
where $\psi_{r,t}(z)=\int_0^z |\tilde{\gamma}_{r,t}'|(s)\,ds$.

This induces the perturbed phase functions
\begin{align*}
  \chi_{r,t} \coloneqq \chi \circ \psi_{r,t}^{-1}\,.
\end{align*}
With the optimal profile $q_\eps$ determined in \eqref{eq:OptProfile} we set
\begin{align*}
  M^*_\eps(s) \coloneqq 
  \begin{cases}
    q_\eps\big(\sdist(\cdot,J)\big)\circ\psi_{r,t} \quad&\text{ if }0<\eps<\eps_0\,,\\
    1 &\text{ for }\eps=0\,.
  \end{cases}
\end{align*}
and define
\begin{align*}
  M_{\eps,r,t} \coloneqq M^*_\eps \circ \psi_{r,t}^{-1}.
\end{align*}
We next introduce a function that measures the deficit in the mass constraints. 
For $L_{r,t}$ denoting the lengths of $\gamma_{r,t}$ we define $\Phi:(-\sqrt[4]{\eps_0},\sqrt[4]{\eps_0})\times (-\delta_0,\delta_0)\times (-\delta_0,\delta_0)\to\R$,
\begin{align*}
  \Phi(\eps,r,t) &\coloneqq 
  \int_0^{L_{r,t}} 
  \begin{pmatrix}
    1 - \chi_{r,t}(s) \\ \chi_{r,t}(s) 
  \end{pmatrix} 
  M_{\eps^4,r,t}(s) \,ds - 
  \begin{pmatrix}
    M_0 \\ M_1 
  \end{pmatrix}\\
  &= \int_0^L 
  \begin{pmatrix}
    1 - \chi(s) \\ \chi(s) 
  \end{pmatrix}
  M^*_{\eps^4}(s) |\tilde{\gamma}_{r,t}'(s)|\,ds - 
  \begin{pmatrix}
    M_0 \\ M_1 
  \end{pmatrix}\,.
\end{align*}
We observe that $\Phi(0,0,0)=0$ and $\Phi$ is continuously differentiable. 
We compute
\begin{equation*}
  |\tilde{\gamma}_{r,t}'|^2 = \big(1+(r\rho+t\sigma)\nu'\cdot\gamma'\big)^2
  + \big(r\rho'+t\sigma'\big)^2\,.
\end{equation*}
Using $|\gamma'|=1$ and $\gamma'$ orthogonal to $\nu$ we deduce
\begin{align*}
  \nabla_{(r,t)}|\tilde{\gamma}_{r,t}'|(0,0) &= \gamma' \cdot 
  \begin{pmatrix}\rho \nu' + \rho' \nu\\ \sigma \nu' + \sigma' \nu\end{pmatrix}
  = (\gamma' \cdot \nu')\begin{pmatrix}\rho\\\sigma\end{pmatrix}
  = \kappa\begin{pmatrix}\rho\\\sigma\end{pmatrix}\,.
\end{align*}
Thus
\begin{align*}
  A := \nabla_{(r,t)} \Phi(0,0,0) 
  = \int_0^L 
  \begin{pmatrix} 
    (1 - \chi) \rho & (1- \chi) \sigma \\ 
    \chi \rho & \chi\sigma 
  \end{pmatrix}(s) 
  \kappa(s)\, ds.
\end{align*}
We now choose smooth $L$-periodic functions $\rho,\sigma$ such that for some $\delta>0$
\begin{align*}
  \spt\rho\cap [0,L] &\subset \{\chi=0\}_{-\delta}\,,\quad
  \int_0^L \rho\kappa\,ds\neq 0\,,\\
  \spt\sigma\cap [0,L] &\subset \{\chi=1\}_{-\delta}\,,\quad
  \int_0^L \sigma\kappa\,ds\neq 0\,,
\end{align*}
where for $E\subset \R$ we let $E_{-\delta}:= \{\dist(\cdot,E^c)>\delta\}$.

With this choices, we see that $A$ is invertible. 
By the implicit function theorem there exist differentiable functions $r$ and $t$ with $r(0)=t(0)=0$ and $\Phi\big(\eps,r(\eps),t(\eps)\big)=0$ for $0<\eps<\eps_0$ sufficiently small. 

\item We then set $r_k=r(\sqrt[4]{\epsk})$, $t_k=t(\sqrt[4]{\epsk})$ and define $Z_k=(\gamma_k,\theta_k,\chi_k,M_k)$,
\begin{align*}
  \gamma_k \coloneqq \gamma_{r_k,t_k}\,,\quad
  \chi_k \coloneqq \chi_{r_k,t_k}\,,\quad
  M_k \coloneqq M_{\epsk,r_k,t_k}\,,\quad 
  \theta_k \coloneqq \nu_k := (\gamma_k')^{\perp}\,.
\end{align*}
Then $Z_k\in \calZ_{\epsk}$ and we have with $L_k:=L_{r_k,t_k}$
\begin{align*}
  M_1(Z_k)-M_1= 
  \int_0^{L_k}\big(1 - \chi_{r_k,t_k}(s)\big)M_{\epsk,r_k,t_k}(s) \,ds -M_1
  = \Phi_1(\sqrt[4]{\epsk},r(\sqrt[4]{\epsk}),t(\sqrt[4]{\epsk})) =0\,,
\end{align*}
and analogously $M_2(Z_k)-M_2=0$, hence the mass constraint \eqref{eq:MassConstraint} is satisfied.

Since $\rho$ and $\gamma$ are smooth we have $\gamma_k\to \gamma$ in $C^{\infty}$. 

\item To compute the limit of the phase separation energy $E_{\epsk}$ from \eqref{eq:DefPEps} we set $a_k=a_{\epsk}$, $\lambda_k=\lambda_{\epsk}$ and recall that
\begin{align}
  E_{\epsk}(Z_k) &= \int_0^{L_k} \Big[ \epsk^{-2}(1-a_k(\chi_k(s)) M_k(s))^2 +\frac{1}{2}|M_k'|^2(s)\Big]\,ds\nonumber\\
  &= \int_0^L \Big[\epsk^{-2}(1-a_k(\chi(s)) M^*_{\epsk}(s))^2 +\frac{1}{2|\tilde\gamma_k'|^2(s)}|(M^*_{\epsk})'|^2(s)\Big]|\tilde\gamma_k'|(s)\,ds\,,
\label{eq:EepsZk}
\end{align}
where we have set $\tilde\gamma_k:=\tilde{\gamma}_{r_k,t_k}$ and used a variable transformation.

We observe that $\rho=\sigma=0$ in the $\delta$-neighborhood $J_\delta$ of $J$, $J_\delta:=\{\dist(\cdot,J)<\delta\}$, which yields $\tilde{\gamma}_k=\gamma$ and $|\tilde\gamma_k'|=1$ in $J_\delta$.
Therefore, the contribution of $J_\delta$ to the integral on the right-hand side of \eqref{eq:EepsZk} can then be computed as in the proof of Theorem \ref{simp:gamma}, which gives
\begin{align}
  \int_{J_\delta \cap (0,L)} \Big[\epsk^{-2}\big(1-a_k(\chi(s)) M^*_{\epsk}(s)\big)^2 +\frac{1}{2}|(M^*_{\epsk})'|^2(s)\Big]\,ds\,
  = \frac{c^2}{\sqrt{8}}|J\cap (0,L)| + \omega_\delta(k)\,,
\label{eq:EepsZk2}
\end{align}
where $\omega_\delta(k)\to 0$ with $k\to\infty$.

Again by computations that are analogous to those in the proof of Theorem \ref{simp:gamma} we have on $([0,L) \setminus J_{\delta})$ that
\begin{equation*}
  \epsk^{-2} (1- a_k(\chi_k) M_k)^2+ \frac{1}{2}|M_k'|^2 \leq C\epsk^{-1}\exp\big(-\sqrt{8}\delta\epsk^{-1}\big)
\end{equation*}
and therefore exponentially small with $k\to\infty$.
Since $|\tilde\gamma_k'|\leq 2$ for $k$ sufficiently small we deduce from \eqref{eq:EepsZk} and \eqref{eq:EepsZk2} that
\begin{align}
  \lim_{k\to\infty} E_{\epsk}(Z_k) &= \frac{c^2}{\sqrt{8}}|J\cap (0,L)| = \frac{c^2}{\sqrt{8}}|\nabla \chi|([0,L))\,.
\label{eq:EepsZk3}
\end{align}

\item We next characterize the limit of the bending energy contribution $G_\eps$. 
Since $\theta_k=\nu_k$ this term reduces to
\begin{align*}
  G_{\epsk}(Z_k) = \int_0^{L_k} \frac{a_k(\chi_k)^2}{4}|\nu_k'|^2M_k^4\,ds\,.
\end{align*}
Using that both $a_k(\chi_k)$ and $M_k$ uniformly converge to $1$ we deduce that 
\begin{align*}
  \int_0^{L_k} \frac{a_k(\chi_k)^2}{4}|\nu_k'|^2M_k^4\,ds = \int_0^{L_k} \frac{|\gamma_k''|^2}{4} \,ds +\omega(k)\,,
\end{align*}
where $\omega(k)\to 0$ for $k\to\infty$.

Since $\gamma_k$ converges to $\gamma$ in $C^{\infty}$ we therefore obtain that
\begin{align*}
  \lim_{k\to\infty} G_{\epsk}(Z_k) &= \frac{1}{4}\int_0^L |\gamma''|^2 \,ds\,.
\end{align*}
Together with \eqref{eq:EepsZk3} this yields \eqref{eq:Limsup} under the additional assumption that $\gamma$ is smooth.
\item Now let $\gamma$ be general. 
Then we can approximate $\gamma$ strongly in $H^2_{\loc}$ by a sequence $(\gamma^{(\ell)})_{\ell\in\N}$ of $L$-periodic smooth functions parametrized by arclength with $|\{|(\gamma^{(\ell)})''|=0\}|=0$.
Setting $\chi^{(\ell)}=\chi$ we have
\begin{equation*}
  \calE(\gamma^{(\ell)},\chi^{(\ell)}) \to \calE(\gamma,\chi)\quad\text{ as }\ell\to\infty\,.
\end{equation*}
As shown above for each $\ell\in\N$ there exists a recovery sequence $(Z_k^{(\ell)})_{k}$ for $(\gamma^{(\ell)},\chi^{(\ell)})$.
Choosing a diagonal sequence we obtain a recovery sequence for $(\gamma,\chi)$ with the required properties.
\end{enumerate}
\end{proof}

\section{Discussion}\label{sec:5}
We have proposed a simple continuum model for membranes made of two different kinds of amphiphilic molecules.
The model is characterized by an energy that involves a small length scale parameter $\eps>0$, which eventually can be linked to the thickness of low energy configurations.
The main question behind this contribution is if such a model (possibly under additional restrictions) shows a preference for structures that resemble phase separated membranes in the following two respects: Firstly that preferred structures are thin in one and extended in the other directions and that preferred structures consist of nearly homogeneous regions with respect to the lipid composition and a thin transition layer in between.
And secondly, that energy contributions can be identified that are linked to stretching and bending of the overall macroscopic structure and to the size of the transition layer between the nearly homogeneous phases.
Within simplified settings we (at least partially) investigate both properties by studying the limit $\eps\to 0$.
We thereby use variational analysis, in particular $\Gamma$-convergence techniques.

The model itself allows to compare a large class of possible configurations, which might approach different kinds of limit structures. 
This in principle allows to verify self-aggregation properties of multi-phase amphiphiles structures in liquid environments and to justify postulates that are imposed in macroscale models.

\smallskip
The model proposed in this paper carries a Cahn--Hilliard--Van der Waals type contribution and for this part we have presented a rather complete mathematical justification of the asymptotic reduction to the perimeter functional, which is the fundamental energy description of macroscale phase separation.
Regarding the full mesoscale energy we have presented some partial results for the macroscopic scale transition.
The construction of a recovery sequence gives good indication that the mesoscale energy in fact characterizes key features of amphiphilic multiphase membranes.

\smallskip
In the construction of a recovery sequence for the full mesoscale model the main restriction is that we only consider a single curve, which somehow corresponds to a single monolayer of amphiphiles.
Generalization to collections of monolayers and extensions to bilayer structures even of the recovery sequence construction become by the multiphase character challenging.
One difficulty can be motivated when one starts from a kind of mid-surface.
A phase transition enforces a non-homogeneity of the thickness, which for the upper and lower surfaces induces a change in geometry and an excess energy, which needs to be avoided.
Another major challenge is to describe the phase separation of limit structures, since the two phases might (even in the limit) be differently distributed on mesoscale layers that have overlapping support in the limit.


\smallskip
Let uns finally comment on some interpretations of the results of this paper in view of the applications that motivated the model.
One interesting observation of the presented analysis is that the particular choice of scaling of the parameter $\lambda_\eps$, which expresses a length difference of order $\eps^{3/2}$ between the two types of lipids, leads to a bending and to a phase interface contribution \emph{on the same order}.
Also, the thickness of the mesoscale layers and the width of the transition region between the phases is of the same order $\eps$.
Different choices for the scaling of $\lambda_\eps$ and other regimes could be considered and an evaluation of the emerging behavior is up to future research.

Our contribution shows that a mathematical analysis as initiated in this paper can characterize specific properties induced by the modeling choices.
The approach might thereby contribute to the discussion of necessary ingredients for a minimal multi-phase membrane that shows the «correct» qualitative behavior.

%

\printbibliography
\end{document}